\documentclass[11pt,reqno]{amsart}
\usepackage{ytableau,tikz,hyperref}
\usepackage{makecell}
\usepackage[backend=bibtex, style=alphabetic, maxnames=10]{biblatex}
\usepackage{palatino}
\usepackage{mathtools}
\usepackage[left=3cm,right=3cm,top=3cm,bottom=3cm]{geometry}
\usepackage[shortlabels,inline]{enumitem}

\newtheorem{theorem}{Theorem}
\newtheorem{remark}{Remark}
\newtheorem{example}{Example}
\newcommand\bremark{\begin{remark}\begin{upshape}}
\newcommand\eremark{\end{upshape}\end{remark}}
\newtheorem{proposition}{Proposition}
\newtheorem{corollary}{Corollary}
\newtheorem{lemma}{Lemma}
\newtheorem{definition}{Definition}

\DeclareMathOperator{\de}{det}
\DeclareMathOperator{\h}{ht}

\allowdisplaybreaks
\title[]{A new proof of the Murnaghan--Nakayama-type rule for stable Grothendieck polynomials of Grassmannian type}
\author[]{Siddheswar Kundu}
\address{Department of Mathematics, Indian Institute of Technology Kanpur, Kanpur, Uttar Pradesh 208016, India.}
\email{kundusidhu96@gmail.com}
\begin{document}
\begin{abstract}
In this article, we present a new proof the Murnaghan--Nakayama-type rule for the stable Grothendieck polynomials of Grassmannian type, which was originally established by Nguyen--Hiep--Son--Thuy.
\end{abstract}
\maketitle
\section{Introduction}
The concept of Grothendieck polynomials, which serve as K-theoretic analogues of Schubert polynomials, was introduced by Lascoux and Sch\"utzenberger~\cite{Lascoux:G1}. Later, Fomin and Kirillov~\cite{Fomin:Yang-Baxter} defined a parameterized version, the $\beta$-Grothendieck polynomials, and investigated their stable limits. The stable Grothendieck polynomials of Grassmannian type $G^{\beta}_{\lambda}(x_1,\dots,x_n)$, which are indexed by partitions, serve as the K-theoretic analogs of Schur polynomials $s_{\lambda}$ and form a basis for (a completion of) the symmetric function space. Yeliussizov~\cite{Yeliussizov:canonical} further extended this family to a two-parameter version, calling them canonical stable Grothendieck functions $G^{(\alpha,\beta)}_{\lambda}(x_1,\dots,x_n)$.

Schur polynomials play a significant role in the representation theory of general linear groups and symmetric groups. They are the characters of finite-dimensional irreducible polynomial representations of general linear groups. Furthermore, Schur polynomials form a crucial basis for the algebra of symmetric functions, alongside other sets like the power sum symmetric functions. The classical Murnaghan--Nakayama rule~\ref{MN:classical} provides the formula for expanding the product of a Schur function $s_{\lambda}$ with a power sum symmetric function $p_k$ as a linear combination of Schur functions. 
Murnaghan--Nakayama rules exist for various other mathematical settings. For example,
\begin{itemize}
    \item A plethystic version is detailed in \cite{MN:Plethysm}.
    \item A rule for non-commutative Schur functions can be found in \cite{MN:non-comu}.
    \item In \cite{MN:k-Schur}, a Murnaghan--Nakayama-type rule for k-Schur functions is presented. 
\end{itemize}
The purpose of this note is to present another proof (see \S\ref{sec:new proof}) of the Murnaghan--Nakayama-type rule for $G^{\beta}_{\lambda}(x_1,\dots,x_n)$, established in \cite{Murnaghan-Nakayama:stable}. Our proof strategy is directly inspired by the structure of the classical Murnaghan--Nakayama rule's proof, detailed in \cite[Theorem 7.17.1]{Stanley:vol.2}. We also obtain a Murnaghan-Nakayama-type rule for $G^{(\alpha,\beta)}_{\lambda}(x_1,\dots,x_n)$ in Corollary~\ref{coro:canonical}.
\section{Preliminaries}
Let $\mathbb{N}=\{1,2,\dots \}$, $\mathbb{Z}_{\geq 0}=\{0,1,2,\dots \}$. For $ n\in \mathbb{N}$, we use $\mathcal{P}[n]$ to represent the set of partitions with at most $n$ non-zero parts, i.e., the set consists of integer sequences $\lambda=(\lambda_1 \geq  \dots \geq \lambda_n \geq 0)$.
We can visually represent a partition $\lambda$ using its Young diagram $Y(\lambda)$, which is a collection of boxes that are top and left justified and the $i^{th}$ row contains $\lambda_i$ boxes. For partitions $\nu$ and $\lambda$, such that $ \lambda_i \leq \nu_i $ $ \forall i \geq 1$, the skew shape $\nu/\lambda$ is formed by taking the set theoretic difference $Y(\nu)-Y(\lambda)$ and its number of boxes, rows, columns are denoted by  $|\nu/\lambda|$, $r(\nu/\lambda)$, $ c(\nu/\lambda)$ respectively. The skew shape $(5,3,1)/(2,1)$ is shown below.
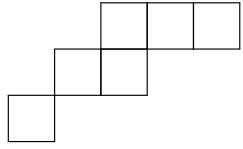
\begin{center}
  \ytableausetup{nosmalltableaux,boxsize=0.6 cm}
\begin{ytableau} 
\none & \none & \null & \null & \null \\
\none & \null & \null \\
\null \\
\end{ytableau}  
\end{center}
We call two boxes in a skew shape \emph{adjacent} if they share an edge. A skew shape is said to be \emph{connected} if every pair of boxes within the shape is connected by a sequence of adjacent boxes contained in the shape. A \emph{ribbon} is a special type of connected skew shape, which is defined by the absence of any $2 \times 2$ square. Let $\mathcal{R}[t]$ denote the set of all ribbons with $t$ boxes. The \emph{height} $\h(\nu/\lambda)$ of a ribbon $\nu/\lambda$ equals the number of non-empty rows minus one. The maximal ribbon along the northwest border of a connected skew Young diagram $\nu/\lambda$ is the largest possible ribbon $\mu/\lambda$ lying within $\nu/\lambda$.

We set $X^n =(x_1,x_2,\dots,x_n)$ to be an $n$-tuple of commuting indeterminates. Let $\Lambda$ denote the ring of symmetric functions in infinitely many variables $x_1,x_2,\dots$, and $\widehat{\Lambda}$ its completion over $\mathbb{Z}[\alpha, \beta]$ for two generic parameters $\alpha$ and $\beta$. For each $n \ge 1$, let $\widehat{\Lambda}_n$ be the subring of $\widehat{\Lambda}$ obtained by specializing $x_{n+i} = 0$ for all $i \ge 1$, restricting the functions to finitely many variables $x_1, \dots, x_n$. We also let $\Lambda_n$ be the subring of $\Lambda$ consisting all symmetric polynomials in $x_1,\dots,x_n$. 
\begin{definition}
\begin{upshape}
 \cite[Definition 3.1]{Yeliussizov:canonical}   
\end{upshape}
For $\lambda \in \mathcal{P}[n]$, we define the \emph{canonical stable Grothendieck function} $G_{\lambda}^{(\alpha,\beta)}(X^n)$ by the formula below
$$ G_{\lambda}^{(\alpha,\beta)}(X^n):= \frac{\de \Bigg [\frac{x_i^{\lambda_j+n-j} 
(1+\beta x_i)^{j-1}}{(1-\alpha x_i)^{\lambda_j}}\Bigg]_{1 \leq i,j \leq n}}{\de \Big[ x_i^{n-j}\Big]_{1 \leq i,j \leq n} }.$$
\end{definition}
Since the expression above is the ratio of two skew-symmetric expressions, $G_{\lambda}^{(\alpha,\beta)}(X^n)$ is a well-defined symmetric power series in $x_1, \dots, x_n$, and thus belongs to $\widehat{\Lambda}_n$. Furthermore, $G_{\lambda}^{(\alpha,\beta)}(x_1, \dots, x_n) = 0$ for any partition $\lambda$ with number of non-zero parts greater than $n$. We highlight the following specializations.
\begin{enumerate}
    \item $G_{\lambda}^{(0,\beta)}(X^n)$ coincides with the stable Grothendieck polynomial of Grassmannian type $G_{\lambda}^{\beta}(X^n) \in \Lambda_n$ (see \cite[Corollary 1]{K-Bender-Knuth}). While the stable Grothendieck polynomial $G^{\beta}_{\lambda}(x_1,x_2,\dots)$ is often defined for infinitely many variables in the literature, this paper considers only finitely many variables.
    \item $G_{\lambda}^{(0,0)}(x_1,x_2,\dots,x_n)$ is equal to the Schur polynomial $s_{\lambda}(x_1,x_2,\dots,x_n)$.
\end{enumerate}
\bremark
\label{remark:main}
\begin{upshape}
    \cite[Proposition 3.4]{Yeliussizov:canonical}
\end{upshape}
$G_{\lambda}^{(\alpha,\beta)}(X^n)= G_{\lambda}^{(0,\alpha + \beta)}(\frac{x_1}{1-\alpha x_1}, \frac{x_2}{1-\alpha x_2},\dots,\frac{x_n}{1-\alpha x_n})$.
\eremark
\begin{definition}
For $r \geq 1$, the $r^{th}$ \emph{power sum symmetric function} $p_r(X^n)$ is defined as follows:
\[ 
 p_r(X^n):= \displaystyle\sum_{j=1}^{n} x_j^{r}.
\]
\end{definition}
\section{Murnaghan--Nakayama type rule}
\label{Sec:3}
\subsection{Murnaghan--Nakayama rule for Schur polynomials}
\label{MN:classical}
Let $k \in \mathbb{N}, \lambda \in \mathcal{P}[n]$. Then the classical Murnaghan--Nakayama rule \cite[Theorem 7.17.1]{Stanley:vol.2} states
$$p_k(X^n)s_{\lambda} (X^n)=\displaystyle\sum_{\nu : \nu/\lambda \in \mathcal{R}[k]}(-1)^{\h(\nu/\lambda)}s_{\nu}(X^n).$$
\begin{example}
Consider $\lambda=(2,1) \in \mathcal{P}[5]$, $k=3$. We display below the partitions arising in the expansion of $p_3 s_{(2,1)}$, with the ribbons highlighted in yellow. 
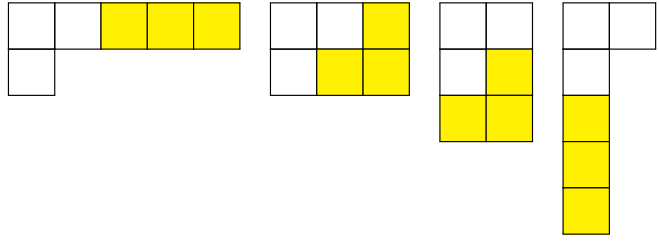
\begin{center}
  \ytableausetup{nosmalltableaux,boxsize=0.6 cm}
\begin{ytableau} 
\null & \null & *(yellow) \null & *(yellow) \null & *(yellow) \null \\
\null \\
\end{ytableau} 
\hspace{0.2cm}
\begin{ytableau} 
\null & \null & *(yellow) \null \\
\null & *(yellow) \null & *(yellow) \null\\
\end{ytableau}
\hspace{0.2cm}
\begin{ytableau} 
\null & \null  \\
\null & *(yellow) \null \\
*(yellow) \null & *(yellow) \null
\end{ytableau} 
\hspace{0.2cm}
\begin{ytableau} 
\null & \null  \\
\null \\
*(yellow) \null \\
*(yellow) \null \\
*(yellow) \null
\end{ytableau} 
\end{center}
Thus we have $$ p_3 s_{(2,1)}=s_{(5,1)} -s_{(3,3)}-s_{(2,2,2)} + s_{(2,1,1,1,1)}.$$ 
\end{example}
\subsection{Murnaghan--Nakayama-type rule for stable Grothendieck polynomials of Grassmannian type}
\label{MN:stable}
The theorem below provides a Murnaghan--Nakayama-type rule for the stable Grothendieck polynomials of Grassmannian type $G_{\lambda}^{\beta}(X^n)$.
\begin{theorem}
\begin{upshape}
 \cite[Theorem 1.1]{Murnaghan-Nakayama:stable}   
\end{upshape}
\label{theorem:stable}
Given $\lambda \in \mathcal{P}[n]$ and $k \in \mathbb{N},$ we have 
$$ p_k(X^n) G_{\lambda}^{\beta}(X^n)= \displaystyle\sum_{\nu}(-\beta)^{|\nu/\lambda|-k} (-1)^{k-c(\nu/\lambda)} \binom{r(\nu/\lambda)-1}{k-c(\nu/\lambda)} G_{\nu}^{\beta}(X^n),$$
where the sum runs over all partitions $\nu \in \mathcal{P}[n]$ such that $\nu/\lambda$ is connected skew shape, $c(\nu/\lambda) \leq k$, and the maximal ribbon along its northwest border has size at least $k$.
\end{theorem}
\begin{example}
Consider $\lambda=(3,2,1) \in \mathcal{P}[3]$, $k=3$. Then all $\nu$ such that $G_{\nu}^{\beta}(X^3)$ occurs in the expansion of $p_3(X^3) G_{\lambda}^{\beta}(X^3)$ are shown below, with $\nu/\lambda$ highlighted in yellow.
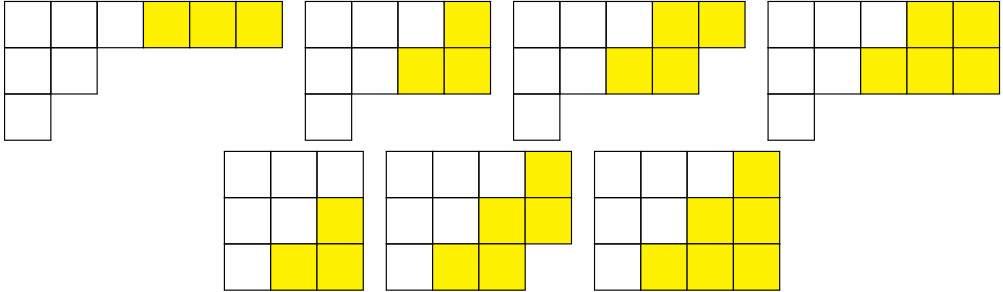
\begin{center}
\ytableausetup{nosmalltableaux,boxsize=0.6 cm}
\begin{ytableau} 
\null & \null & \null & *(yellow) \null & *(yellow) \null & *(yellow) \null \\
\null & \null \\
\null \\
\end{ytableau}
\hspace{0.1cm}
\begin{ytableau} 
\null & \null & \null & *(yellow) \null \\
\null & \null & *(yellow) \null & *(yellow) \null\\
\null \\
\end{ytableau}
\hspace{0.1cm}
\begin{ytableau} 
\null & \null & \null & *(yellow) \null & *(yellow) \null\\
\null & \null & *(yellow) \null & *(yellow) \null\\
\null \\
\end{ytableau}
\hspace{0.1cm}
\begin{ytableau} 
\null & \null & \null & *(yellow) \null & *(yellow) \null\\
\null & \null & *(yellow) \null & *(yellow) \null  & *(yellow) \null\\
\null \\
\end{ytableau} 
\end{center}
\begin{center}
\ytableausetup{nosmalltableaux,boxsize=0.6 cm}
\begin{ytableau} 
\null & \null & \null \\
\null & \null & *(yellow) \null\\
\null & *(yellow) \null & *(yellow) \null\\
\end{ytableau}
\hspace{0.1cm}
\begin{ytableau} 
\null & \null & \null & *(yellow) \null  \\
\null & \null & *(yellow) \null & *(yellow) \null \\
\null & *(yellow) \null & *(yellow) \null\\
\end{ytableau}
\hspace{0.1cm}
\begin{ytableau} 
\null & \null & \null & *(yellow) \null  \\
\null & \null & *(yellow) \null & *(yellow) \null \\
\null & *(yellow) \null & *(yellow) \null & *(yellow) \null\\
\end{ytableau} 
\end{center}
Thus, we have
\[p_3G_{\lambda}^{\beta}=G_{(6,2,1)}^{\beta}-G_{(4,4,1)}^{\beta}-\beta G_{(5,4,1)}^{\beta} + \beta^2 G_{(5,5,1)}^{\beta}- G_{(3,3,3)}^{\beta} + \beta^2  G_{(4,4,3)}^{\beta}-\beta^3 G_{(4,4,4)}^{\beta}.\]
Here, we abbreviate $p_3(X^3)$ as $p_3$ and $G_{\lambda}(X^3)$ as $G_{\lambda}$ for the relevant partitions $\lambda$.
\end{example}
As a corollary of Theorem~\ref{theorem:stable}, we can state a Murnaghan--Nakayama-type rule for $G^{(\alpha,\beta)}_{\lambda}$ as follows:
\begin{corollary}
\label{coro:canonical}
Given $k \in \mathbb{N}$, we define $p^{\alpha}_k (x_1,x_2,\dots,x_n):=p_k(\frac{x_1}{1-\alpha x_1},\frac{x_2}{1-\alpha x_2},\dots, \frac{x_n}{1-\alpha x_n})$. Then, for 
$\lambda \in \mathcal{P}[n]$,
$$ p^{\alpha}_k (X^n)G_{\lambda}^{(\alpha, \beta)}(X^n)= \displaystyle\sum_{\nu}(-\alpha - \beta)^{|\nu/\lambda|-k} (-1)^{k-c(\nu/\lambda)} \binom{r(\nu/\lambda)-1}{k-c(\nu/\lambda)} G_{\nu}^{(\alpha, \beta)}(X^n),$$ where $\nu$ varies over as mentioned in Theorem~\ref{theorem:stable}.
\end{corollary}
In \cite{Murnaghan-Nakayama:stable}, the authors proved Theorem~\ref{theorem:stable} by applying Lenart's Pieri rules \cite{Lenart:Comb-aspects} for $G_{\lambda}^{\beta}(X^n)$. Our approach follows the proof structure of the classical Murnaghan--Nakayama rule, as detailed in \cite[Theorem 7.17.1]{Stanley:vol.2}. We are now ready to present the proof.
\subsection{New proof}
\label{sec:new proof}
For $\gamma \in \mathbb{Z}_{\geq 0}^{n}$, define $A_{\gamma}^{\beta}(X^n):=\de\Big(x_i^{\gamma_j}(1+\beta x_i)^{j-1}\Big)_{1 \leq i,j \leq n}$.  For $i \geq 1$, we set $y_i=1+\beta x_i$.
\begin{lemma}
For $\gamma \in \mathbb{Z}_{\geq 0}^{n}$, $p_r(X^n) A_{\gamma}^{\beta}(X^n)=\displaystyle\sum_{j=1}^{n}A_{\gamma+r\epsilon_j}^{\beta}(X^n)$, where $\epsilon_j \in \mathbb{Z}^{n}$ whose $j^{th}$ entry is $1$ and the others are $0$. 
\end{lemma}
\begin{proof}
We prove this lemma by induction on $n$. We first check it for $n=2$.
$$ p_r(x_1,x_2) A_{\gamma}^{\beta}(x_1,x_2) =(x_1^{r}+x_2^{r})\Big(x_1^{\gamma_1}x_2^{\gamma_2}y_2-x_1^{\gamma_2}x_2^{\gamma_1}y_1\Big ) $$
$$=\Big( x_1^{\gamma_1 + r}x_2^{\gamma_2}y_2-x_1^{\gamma_2+r}x_2^{\gamma_1}y_1\Big) + \Big( x_1^{\gamma_1}x_2^{\gamma_2 +r }y_2-x_1^{\gamma_2}x_2^{\gamma_1 + r}y_1 \Big)$$
$$=\Big(x_1^{\gamma_1+r}x_2^{\gamma_2}y_2 - x_1^{\gamma_2}x_2^{\gamma_1 + r}y_1 \Big) + \Big(x_1^{\gamma_1}x_2^{\gamma_2 +r }y_2- x_1^{\gamma_2+r}x_2^{\gamma_1}y_1 \Big) .$$
Thus we have
$$p_r(x_1,x_2) A_{\gamma}^{\beta}(x_1,x_2) = A_{(\gamma_1 +r, \gamma_2)}^{\beta}(x_1,x_2) + A_{(\gamma_1 , \gamma_2 +r)}^{\beta}(x_1,x_2).$$
So the lemma is true for $n=2$. Let the lemma be true for $n=k-1 (k > 2)$ and for $1 \leq i \leq k$, $X^k_{i^{*}}=(x_1,\dots, x_{i-1}, x_{i+1}, \dots, x_k)$. We also let $\bar{\gamma}=(\gamma_1,\gamma_2,\dots,\gamma_{k-1})$. Then we have 
$$p_r(X^k) A_{\gamma}^{\beta} (X^k)=p_r (X^k)
\begin{vmatrix}
x_1^{\gamma_1} & x_1^{\gamma_2} y_1& \cdots & x_1^{\gamma_{k-1}}y_1^{k-2}  & x_1^{\gamma_k}y_1^{k-1}\\
x_2^{\gamma_1} & x_2^{\gamma_2}y_2& \cdots & x_2^{\gamma_{k-1}}y_2^{k-2} & x_2^{\gamma_k}y_2^{k-1}\\
\cdots & \cdots &\cdots &\cdots & \cdots\\
x_k^{\gamma_1} & x_k^{\gamma_2} y_k& \cdots & x_k^{\gamma_{k-1}} y_k^{k-2} & x_k^{\gamma_k} y_k^{k-1}
\end{vmatrix}$$
$ = (x_1^{r} + x_2^{r} + \cdots + x_k^{r}) \bigg( 
\displaystyle\sum_{j=1}^k (-1)^{k+j} x_j^{\gamma_k} y_j^{k-1} A_{\bar{\gamma}}^{\beta}(X^k_{j^{*}}) \bigg).$\\
Since we have assumed the lemma is true for $n=k-1$, we have \\
$p_r(X^k) A_{\gamma}^{\beta} (X^k)$
$= (-1)^{k+1} x_1^{\gamma_k +r }y_1^{k-1} A_{\bar{\gamma}}^{\beta}(X^k_{1^{*}}) + (-1)^{k+1} x_1^{\gamma_k} y_1^{k-1} \displaystyle\sum_{j=1}^{k-1}A_{\bar{\gamma}+r \epsilon_{j}}^{\beta}(X^k_{1^{*}}) $\\
$ + (-1)^{k+2}x_2^{\gamma_k +r}y_2^{k-1} A_{\bar{\gamma}}^{\beta}(X^k_{2^{*}}) + (-1)^{k+2} x_2^{\gamma_k}y_2^{k-1} \displaystyle\sum_{j=1}^{k-1}A_{\bar{\gamma}+r \epsilon_{j}}^{\beta}(X^k_{2^{*}})$\\
$+\cdots + (-1)^{2k} x_k^{\gamma_k +r}y_k^{k-1} A_{\bar{\gamma}}^{\beta}(X^k_{k^{*}}) +  (-1)^{2k} x_k^{\gamma_k}y_k^{k-1} \displaystyle\sum_{j=1}^{k-1}A_{\bar{\gamma}+r \epsilon_{j}}^{\beta}(X^k_{k^{*}})$\\
$= \displaystyle\sum_{j=1}^{k} (-1)^{k+j} x_j^{\gamma_k }y_j^{k-1} A_{\bar{\gamma} + r \epsilon_1}^{\beta}(X^k_{j^{*}}) + \cdots + \displaystyle\sum_{j=1}^{k}(-1)^{k+j} x_j^{\gamma_k }y_j^{k-1} A_{\bar{\gamma} + r \epsilon_{k-1}}^{\beta}(X^k_{j^{*}}) \\ + \displaystyle\sum_{j=1}^{k}(-1)^{k+j} x_j^{\gamma_k +r }y_j^{k-1} A_{\bar{\gamma}}^{\beta}(X^k_{j^{*}}) $ \\
$=A_{\gamma + r \epsilon_1}^{\beta}(X^k) + \cdots + A_{\gamma + r \epsilon_{k-1}}^{\beta}(X^k) + A_{\gamma + r \epsilon_k}^{\beta}(X^k)$.
\end{proof}
Since $G^{\beta}_{\lambda}(X^n) = \frac{A_{\lambda + \delta^n}^{\beta}(X^n)}{A^0_{\delta^n}(X^n)}$, it follows from the above lemma that 
$$ p_k(X^n)G^{\beta}_{\lambda}(X^n)=\displaystyle\sum_{j=1}^{n} \frac{A_{\lambda +\delta^n + k \epsilon_{j}}^{\beta}(X^n)}{A^0_{\delta^n}(X^n)},$$
where $\delta^n=(n-1,n-2,\dots,1,0)$.
Thus, to prove Theorem~\ref{theorem:stable}, it is enough to establish the following proposition. 
\begin{proposition}
\label{Proposition:main}
For $\lambda \in \mathcal{P}[n],$ $1 \leq j \leq n ,$ $$A_{\lambda +\delta^n + k \epsilon_{j}}^{\beta}(X^n) = \displaystyle\sum_{\nu}(-\beta)^{|\nu/\lambda|-k} (-1)^{k-c(\nu/\lambda)} \binom{r(\nu/\lambda)-1}{k-c(\nu/\lambda)} A_{\nu +\delta^n}^{\beta}(X^n),$$ where the sum runs over all partitions $\nu \in \mathcal{P}[n]$ such that $\nu/\lambda$ is connected skew shape,  $c(\nu/\lambda) \leq k$ and the maximal ribbon along its northwest border has size at least $\mathrm{k}$, together with the condition that the bottommost non-empty row of $Y(\nu/\lambda)$ lies in $j^{th}$ row of $Y(\nu)$.    
\end{proposition}
\begin{example}
Consider $\lambda=(3,2,1) \in \mathcal{P}[3]$, $k=3$. Then, for $j=1,2,3$, all partitions $\nu$ such that $A_{\nu + \delta^3}(X^3)$ appears in the expansion of $A_{\lambda +\delta^3 + k \epsilon_{j}}^{\beta}(X^3)$ are displayed below in the $j^{th}$ row, with the bottommost non-empty row of $Y(\nu/\lambda)$ highlighted in yellow.
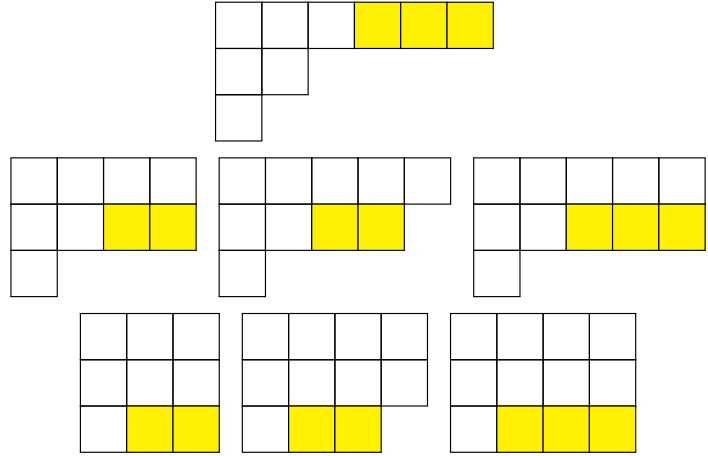
\begin{center}
\ytableausetup{nosmalltableaux,boxsize=0.6 cm}
\begin{ytableau} 
\null & \null & \null & *(yellow) \null & *(yellow) \null & *(yellow) \null \\
\null & \null \\
\null \\
\end{ytableau}
\hspace{0.1cm}
\end{center}
\begin{center}
\ytableausetup{nosmalltableaux,boxsize=0.6 cm} 
\begin{ytableau} 
\null & \null & \null & \null \\
\null & \null & *(yellow) \null & *(yellow) \null\\
\null \\
\end{ytableau}
\hspace{0.1cm}
\begin{ytableau} 
\null & \null & \null &  \null &  \null\\
\null & \null & *(yellow) \null & *(yellow) \null\\
\null \\
\end{ytableau}
\hspace{0.1cm}
\begin{ytableau} 
\null & \null & \null & \null &  \null\\
\null & \null & *(yellow) \null & *(yellow) \null  & *(yellow) \null\\
\null \\
\end{ytableau} 
\end{center}
\begin{center}
\ytableausetup{nosmalltableaux,boxsize=0.6 cm} 
\begin{ytableau} 
\null & \null & \null \\
\null & \null &  \null\\
\null & *(yellow) \null & *(yellow) \null\\
\end{ytableau}
\hspace{0.1cm}
\begin{ytableau} 
\null & \null & \null & \null  \\
\null & \null &  \null &  \null \\
\null & *(yellow) \null & *(yellow) \null\\
\end{ytableau}
\hspace{0.1cm}
\begin{ytableau} 
\null & \null & \null & \null  \\
\null & \null &  \null &  \null \\
\null & *(yellow) \null & *(yellow) \null & *(yellow) \null\\
\end{ytableau}
\end{center}
\end{example}
\begin{proof}
We prove this by induction on $n$. First we check it for $n=2$. It is apparent for the case $j=1$ and the case $j=2$, if $\lambda_1 + 1 > \lambda_2 +k$. When $\lambda_1 +1 = \lambda_2 + k$, it is easy to verify that $ A_{(\lambda_1 +1, \lambda_2 +k)}^{\beta}(x_1,x_2) = (-\beta)A_{\nu +\delta^2}^{\beta}(x_1,x_2) $, where $\nu=(\lambda_1 +1, \lambda_2 + k)$. Now we assume that $\lambda_1 +1 < \lambda_2 + k $.
For $p,q \in \mathbb{Z}_{\geq 0}$, we define the following
$$D(p,q):=
  \begin{vmatrix}
x_1^p & x_1^q \\
x_2^p & x_2^q
  \end{vmatrix}.$$
Now $A_{\lambda + \delta^2 + k \epsilon _2}^{\beta}(X^2) = A_{(\alpha, \alpha +t)}^{\beta}(X^2)$, where $\alpha = \lambda_1 +1 , t= \lambda_2 +k -\lambda_1 -1$. Then we have 
$$ A_{(\alpha, \alpha +t)}^{\beta}(x_1,x_2) = 
\begin{vmatrix}
 x_1^{\alpha} & x_1^{\alpha + t}(1 + \beta x_1) \\
 x_2^{\alpha} & x_2^{\alpha + t}(1 + \beta x_2)
\end{vmatrix} =- 
\begin{vmatrix}
 x_1^{\alpha +t} & x_1^{\alpha} \\
 x_2^{\alpha +t} & x_2^{\alpha}
\end{vmatrix} -\beta 
\begin{vmatrix}
 x_1^{\alpha +t +1} & x_1^{\alpha} \\
 x_2^{\alpha +t +1} & x_2^{\alpha}
\end{vmatrix}$$
$$=-D(\alpha + t, \alpha) -\beta D(\alpha + t +1, \alpha) .$$
It is evident that $ A_{(\alpha +t, \alpha+i)}^{\beta}(x_1,x_2) = D(\alpha +t, \alpha+i) + \beta  D(\alpha +t, \alpha+i+1)$, for $i \geq 0$. Then we have
$$ D(\alpha +t, \alpha) = A_{(\alpha +t, \alpha)}^{\beta}(X^2) - \beta  D(\alpha +t, \alpha+1).$$
Thus we obtain 
$$D(\alpha +t, \alpha) = A_{(\alpha +t, \alpha)}^{\beta}(X^2) - \beta\Big( A_{(\alpha +t, \alpha +1)}^{\beta}(X^2) - \beta  D(\alpha +t, \alpha+2)\Big).$$
Continuing this we have $D(\alpha+t,\alpha) = \displaystyle\sum_{j=0}^{t-1}(-\beta)^j A_{(\alpha +t, \alpha +j)}^{\beta}(x_1,x_2)$. 
Thus $$A_{(\alpha , \alpha +t)}^{\beta}(x_1,x_2) = - \displaystyle\sum_{j=0}^{t-1}(-\beta)^j A_{(\alpha +t, \alpha +j)}^{\beta}(x_1,x_2) -\beta\displaystyle\sum_{j=0}^{t}(-\beta)^j A_{(\alpha +t+1, \alpha +j)}^{\beta}(x_1,x_2).$$
Thus we have
\begin{equation}
    \label{eq:main}
A_{\lambda + \delta^2 + k \epsilon_2}^{\beta}= \displaystyle\sum_{j=0}^{\lambda_2 +k -\lambda_1-2}(-\beta)^j (-1)A_{(\lambda_2+k, \lambda_1 +j +1)}^{\beta} +\displaystyle\sum_{j=0}^{\lambda_2 +k -\lambda_1-1}(-\beta)^{j+1} A_{(\lambda_2+k+1, \lambda_1 +j +1)}^{\beta}. 
\end{equation}
Recall that we are in the case where $\lambda_1 +1 < \lambda_2+k$. Now for $\nu \in \mathcal{P}[2]$, suppose $\nu/\lambda$ is connected skew shape such that $c(\nu/\lambda) \leq k$ and the maximal ribbon along its northwest border has size at least $\mathrm{k}$, together with the condition that the bottommost non-empty row of $Y(\nu/\lambda)$ lies in the second row of $Y(\nu)$. Then we show that $\nu$ has the following two forms:
$$\nu^1=(\lambda_2+k-1,\lambda_1+j+1) \text{ for } 0 \leq j \leq \lambda_2 +k -\lambda_1-2,$$
$$\nu^2=(\lambda_2+k,\lambda_1+j+1) \text{ for } 0 \leq j \leq \lambda_2 +k -\lambda_1-1 .$$
Since $\lambda_1 +1 < \lambda_2+k$, it follows that $\nu^1_1 > \lambda_1$, $\nu^2_1 > \lambda_1 $. Furthermore, the inequalities $\nu^1_2 > \lambda_2$, $\nu^2_2 > \lambda_2$ follow from $\lambda_1 \geq \lambda_2$. Also, we have 
$$c(\nu^1/\lambda)= (\nu^1_2 - \lambda_2)+(\nu^1_1-\nu^1_2)=k-1 ,$$
$$c(\nu^2/\lambda)= (\nu^2_2 - \lambda_2)+(\nu^2_1-\nu^2_2)=k .$$
Since $\nu^1_2 > \lambda_1$ (resp. $\nu^2_2 > \lambda_1$), it follows that $\nu^1/\lambda $ (resp. $\nu^2/\lambda $) is connected. Moreover, $\nu^1_2 > \lambda_2$ (resp. $\nu^2_2 > \lambda_2$) implies that the bottommost non-empty row of $Y(\nu^1/\lambda)$ (resp. $Y(\nu^2/\lambda)$) lies in second row of $Y(\nu^1)$ (resp. $Y(\nu^2)$). The largest possible ribbon $\mu^1/\lambda$ lying within $\nu^1/\lambda$ has size $k$, and is given by
\[
\mu^1=(\lambda_2+k-1,\lambda_1+1).
\]
Similarly, the largest ribbon $\mu^2/\lambda$ contained in $\nu^2/\lambda$ is of size $k+1$, where
\[
\mu^2=(\lambda_2+k,\lambda_1+1).
\]
Thus $\nu^1$, $\nu^2 \in \mathcal{P}[2]$ satisfy the conditions mentioned in Proposition~\ref{Proposition:main} for $j=2$.  

Now we have $A_{(\lambda_2+k, \lambda_1 +j +1)}^{\beta} = A_{\nu^1 + \delta^2}$ and $A_{(\lambda_2+k+1, \lambda_1 +j +1)}^{\beta} = A_{\nu^2 + \delta^2}$. Therefore, by \eqref{eq:main}, we can say that the proposition is true for $n=2$.

Let the proposition be true for $n=r$.
Consider a partition $\lambda \in \mathcal{P}[r+1]$. Then it is enough to prove the proposition for $1 \leq j \leq r$. Let $\lambda^*=(\lambda_1+1,\lambda_2+1,\dots,\lambda_r+1)$ and fix $1 \leq j \leq r$. We recall that $X^{r+1}_{t^{*}}=(x_1,\dots, x_{t-1}, x_{t+1}, \dots, x_{r+1})$ for $1 \leq t \leq r+1$. Then, expanding with respect to $(r+1)^{th}$ column, we obtain $A_{\lambda + \delta^{r+1} + k \epsilon_j}^{\beta}(X^{r+1})$ equals 
$$ A_{(\lambda^* + \delta^r + k \epsilon_j, \lambda_{r+1})}^{\beta}(X^{r+1})= \displaystyle\sum_{t=1}^{r+1} (-1)^{r+1+t} x_t^{\lambda_{r+1}}(1+\beta x_t)^{r} A^{\beta}_{\lambda^* + \delta^r +k\epsilon_j}(X^{r+1}_{t^*}).$$
Since we have assumed the propositon is true for $n=r$, $ A_{\lambda + \delta^{r+1} + k \epsilon_j}^{\beta}(X^{r+1})$ equals
$$ \displaystyle\sum_{t=1}^{r+1} (-1)^{r+1+t} x_t^{\lambda_{r+1}}(1+\beta x_t)^{r} \Bigg ( \displaystyle\sum_{\nu^*}(-\beta)^{|\nu^*/\lambda^*|-k} (-1)^{k-c(\nu^*/\lambda^*)} \binom{r(\nu^*/\lambda^*)-1}{k-c(\nu^*/\lambda^*)} A_{\nu^* +\delta^r}^{\beta}(X^{r+1}_{t^*})\Bigg),$$ where $\nu^*$ varies in the same way as $\nu$ in Proposition~\ref{Proposition:main}.

Thus $A_{\lambda + \delta^{r+1} + k \epsilon_j}^{\beta}(X^{r+1})$ equals 
$$ \displaystyle\sum_{\nu^*}(-\beta)^{|\nu^*/\lambda^*|-k} (-1)^{k-c(\nu^*/\lambda^*)} \binom{r(\nu^*/\lambda^*)-1}{k-c(\nu^*/\lambda^*)} \Bigg ( \displaystyle\sum_{t=1}^{r+1} (-1)^{r+1+t} x_t^{\lambda_{r+1}}(1+\beta x_t)^{r} A_{\nu^* +\delta^r}^{\beta}(X^{r+1}_{t^*}) \Bigg ).$$
Hence, we have
$$A_{\lambda + \delta^{r+1} + k \epsilon_j}^{\beta}(X^{r+1}) = \displaystyle\sum_{\nu^*}(-\beta)^{|\nu^*/\lambda^*|-k} (-1)^{k-c(\nu^*/\lambda^*)} \binom{r(\nu^*/\lambda^*)-1}{k-c(\nu^*/\lambda^*)} A_{(\nu^* +\delta^r,\lambda_{r+1})}^{\beta}(X^{r+1}).$$
We let $\nu=(\nu^*-1,\dots, \nu^*_r -1, \lambda_{r+1})$. Then we have 
$$(\nu^* +\delta^r, \lambda_{r+1})=(\nu^*_1+r-1,\nu^*_2+r-2,\dots,\nu^*_r,\lambda_{r+1})=\nu + \delta^{r+1}.$$
It is clear that $\nu^*/\lambda^* $ and $\nu/\lambda$ represent the same skew shape. So the proposition is true for $n=r+1$.
\end{proof}
\section*{Acknowledgement} 
This work was supported by DAE Plan project RIN 4001 at NISER Bhubaneswar.
\printbibliography
\end{document}